\documentclass[reqno]{amsart}
\usepackage[backend=bibtex, sorting=none, bibstyle=alphabetic, citestyle=alphabetic, sorting=nyt, maxbibnames=99, giveninits=true, isbn=false, url=false, doi=false, eprint=false]{biblatex}  
\renewbibmacro{in:}{}
\bibliography{references.bib}
\usepackage{amssymb, calrsfs, graphics, graphicx, enumerate, enumitem, url, xcolor, hyperref, mathrsfs, listings, comment}
\usepackage{tabularx}
\usepackage[ruled, lined, linesnumbered, longend]{algorithm2e}
\usepackage[justification=centering]{caption}
\newtheorem{theorem}{Theorem}[section]
\newtheorem{lemma}[theorem]{Lemma}

\numberwithin{equation}{section}

\theoremstyle{remark}
\newtheorem*{remark}{Remark}

\lstdefinestyle{CStyle}{
    basicstyle=\footnotesize,
    breakatwhitespace=false,         
    breaklines=true,                 
    captionpos=b,                    
    keepspaces=true,                 
    numbers=left,                    
    numbersep=5pt,                  
    showspaces=false,                
    showstringspaces=false,
    showtabs=false,                  
    tabsize=2,
    language=C
}

\title[Explicit zero density estimate near unity]{Explicit zero density estimate near unity}
\author[C. Bellotti]{Chiara  Bellotti}
\address{School of Science\\
The University of New South Wales, Canberra, Australia}
\email{c.bellotti@adfa.edu.au}
\keywords{Riemann zeta function, zero-density estimates, explicit results}
\subjclass[2020]{Primary 11M06, 11M26. Secondary 11Y35}
\begin{document}
\maketitle
\begin{abstract}
We will provide the first explicit zero-density estimate for $\zeta$ of the form $N(\sigma,T)\le \mathcal{C}T^{B(1-\sigma)^{3/2}}(\log T)^C$. In particular, we improve $C$ to $10393/900=11.547\dots.$ 
\end{abstract}
\section{Introduction}
Let $\zeta(s)$ be the Riemann zeta function and $\rho=\beta+i\gamma$ a non-trivial zero of $\zeta$, with $0<\operatorname{Re}(\rho)<1$. Given $1/2<\sigma<1$, $T>0$, we define the quantity$$
N(\sigma, T)=\#\{\rho=\beta+i \gamma: \zeta(\rho)=0,0<\gamma<T,\  \sigma<\beta<1\},
$$
which counts the number of non-trivial zeros of $\zeta$ with real part greater than the fixed value $\sigma$. Finding upper bounds for $N(\sigma,T)$, commonly known as zero-density estimates, is a problem that has always attracted much interest in Number Theory. If the Riemann Hypothesis is true, then we have $N(\sigma,T)=0$ for every $1/2<\sigma<1$, since all the non-trivial zeros of $\zeta$ would lie on the half-line $\sigma=1/2$. In 2021, Platt and Trudgian \cite{platt_riemann_2021} verified that the Riemann Hypothesis is true up to height $3\cdot 10^{12}$ and, consequently, $N(\sigma,T)=0$ for every $1/2<\sigma<1$, $T\le 3\cdot 10^{12}$. This allow us to restrict the range of values of $T$ for which we aim to estimate $N(\sigma,T)$ to $T>3\cdot 10^{12}$.\\
Different methods have been used by several mathematicians throughout the years, depending on the range for $\sigma$ inside the right half of the critical strip in which we are working. In 1937 Ingham \cite{ingham1937difference} proved that, assuming that $\zeta\left(\frac{1}{2}+i t\right)\ll t^{c+\epsilon}$, one has $
N(\sigma, T)\ll T^{(2+4 c)(1-\sigma)}(\log T)^5
$. In particular, the Lindelöf Hypothesis $\zeta\left(\frac{1}{2}+i t\right)\ll t^\epsilon$ implies  $N(\sigma, T)\ll T^{2(1-\sigma)+\epsilon}$, known as the Density Hypothesis. Ingham's type estimate has been made explicit by Ramaré \cite{ramare2016explicit} and improved in 2018 by Kadiri, Lumley and Ng \cite{KADIRI201822}. Ingham's method is particular powerful when $\sigma$ is close to the half-line, as the upper bound for $\zeta(\frac{1}{2}+it)$ is involved.\\
When $\sigma$ is really close to $1$ and $T$ sufficiently large, estimates of the type 
\begin{equation}\label{ivicest}
    N(\sigma,T)\ll T^{B(1-\sigma)^{3/2}}(\log T)^C
\end{equation}become sharper. Indeed, for $\sigma$ sufficiently close to $1$, and hence $T$ sufficiently large, Richert's bound \cite{richert_zur_1967} $|\zeta(\sigma+it)|\le A |t|^{B(1-\sigma)^{3/2}}(\log |t|)^{2/3}$ is the sharpest known upper bound for $\zeta$. However, contrary to Ingham's type zero-density estimate, the estimate \eqref{ivicest} has never been made explicit. The best non-explicit zero-density estimate of this form is due to Heath-Brown \cite{heathbrown_new_2017} (actually, this result has been recently slightly improved by Trudgian and Yang \cite{trudgian2023optimal}, as they found new optimal exponent pairs; see also \cite{pintz2023remark} for more recent results), in which he used a slightly different bound for $\zeta$, i.e. $\zeta(\sigma+it)\ll t^{B(1-\sigma)^{3/2}+\epsilon}$, where the exponent $B$ turns out to be much smaller ($B<\frac{1}{2}$) than the currently best-known value for $B$ in Richert's bound, that is $B=4.43795$ \cite{bellotti2023explicit}. However, differently from Richert's bound that is fully explicit, the bound used by Heath-Brown cannot be used to get an explicit version for the zero-density estimate of the form \eqref{ivicest}, since the factor $t^{\epsilon}$ cannot be made explicit (see \cite{Ivic2003TheRZ, montgomery_topics_1971} for previous non-explicit results).\medskip\\

The aim of this paper is to provide the first explicit zero-density estimate for $\zeta$ of the form \eqref{ivicest}. Combining Ivić's zero-detection method \cite{Ivic2003TheRZ} with Richert's bound and an explicit zero-free region for $\zeta$, we will prove the following result.
\begin{theorem}\label{theorem1general}
For every $\sigma\in[0.98,1]$ and $T\ge 3$, the following zero-density estimate holds uniformly:
\begin{equation*}
    \begin{aligned}
        &N(\sigma,T)\\&\le (\mathcal{C}_1T^{57.8875(1-\sigma)^{3/2}}(\log T)^{\frac{19703}{1800}}+\mathcal{C}_2T^{33.08(1-\sigma)^{3/2}}(\log T)^{\frac{503}{45}}+0.27(\log T)^{\frac{14}{10}})\log\log T,
    \end{aligned}
\end{equation*}
where  $\mathcal{C}_2=7.65\cdot 10^{10}$ and
\begin{equation*}
    \mathcal{C}_1=\left\{\begin{array}{lll}
    4.68\cdot 10^{23} &  \text{if }3\cdot10^{12}\le T\le e^{46.2},\\ \\
     4.59\cdot 10^{23}   &   \text{if }e^{46.2}< T \le e^{170.2}, \\ \\
        1.45\cdot 10^{23} &   \text{if }e^{170.2}< T\le e^{481958}, \\ \\
        9.77\cdot 10^{21} &   \text{if }T>  e^{481958}.
    \end{array}\right.
\end{equation*}
\end{theorem}
Actually, it is possible to deduce a much simpler upper bound for $N(\sigma,T)$ compared to Theorem \ref{theorem1general}, at the expense of a slightly worse exponent for the factor $(\log T)$.
\begin{theorem}\label{theorem1}
For every $\sigma\in[0.98,1]$ and $T\ge 3$, the following zero-density estimate holds uniformly:\begin{equation}
    N(\sigma,T)\le \mathcal{C}^{\prime}_1T^{57.8875(1-\sigma)^{3/2}}(\log T)^{10393/900},
\end{equation}
where
\begin{equation*}
    \mathcal{C}^{\prime}_1=\left\{\begin{array}{lll}
   2.15\cdot 10^{23} &  \text{if }3\cdot10^{12}< T\le e^{46.2},\\ \\
    1.89\cdot 10^{23}    &   \text{if }e^{46.2}< T \le e^{170.2}, \\ \\
        4.42\cdot 10^{22} &   \text{if }e^{170.2}< T\le e^{481958}, \\ \\
        4.72\cdot 10^{20} &   \text{if } T>  e^{481958}.
    \end{array}\right.
\end{equation*}
\end{theorem}
We observe that the exponent $10393/900=11.547\dots$ for the factor $\log T$ is less than $11.548$, which improves the best-known exponent equal to $15$ due to Ivic \cite{Ivic2003TheRZ}. Furthermore, the value $10393/900$ is the near-optimal one that one can get using Ivić's zero-detection method. Indeed, a factor $(\log T)^3$ comes from the relation for the divisor function in Theorem \ref{michaela-tim}, where the powers of the logarithm are already the optimal ones. Then, Ivić's zero-detection method produces $5$ more powers of $\log T$, as we will see later, that cannot be reduced, when using this method. Finally, another factor $\log T$ comes from the number of zeros in Theorem \ref{thnt}. Hence, all these contributions give already $9$ powers that cannot be removed. The remaining ones come from convergence arguments throughout the proof, where we already tried to optimize our choices as much as possible.
\subsection{Range of $\sigma$ for which Theorem \ref{theorem1} is the sharpest bound.}
As expected, the zero-density estimate found in Theorem \ref{theorem1} becomes powerful when $\sigma$ is very close to $1$, since when $\sigma$ is sufficiently close to $1$ Richert's bound becomes the sharpest one for $\zeta$ inside the critical strip and the Korobov--Vinogradov zero-free region is the widest one. We briefly analyze which is the lowest value of $\sigma$ from which Theorem \ref{theorem1} start being sharper than Kadiri--Lumley--Ng's result \cite{KADIRI201822}, which is of the form 
\begin{equation}\label{kadiri}
    N(\sigma, T) \leq \frac{\mathcal{C}_1}{2 \pi d}(\log (k T))^{2 \sigma}(\log T)^{5-4 \sigma} T^{\frac{8}{3}(1-\sigma)}+\frac{\mathcal{C}_2}{2 \pi d}(\log T)^2,
\end{equation}
where $\mathcal{C}_1,\mathcal{C}_2$ and $d$ are suitable constants defined in Theorem 1.1 of \cite{KADIRI201822}. In order to simplify the notation and the calculus, we rewrite \eqref{kadiri} as
\[
N(\sigma,T)\le CT^{\frac{8}{3}(1-\sigma)}\log^3T,
\]
being $5-4\sigma+2\sigma\rightarrow 3$ for $\sigma$ close to $1$ and the second term in \eqref{kadiri} is included in the constant $C\ge 1$. We want to find $\sigma, T$ such that \begin{equation}\label{mybetter}
    CT^{\frac{8}{3}(1-\sigma)}\log^3T\ge \mathcal{C}_1^\prime T^{B(1-\sigma)^{3/2}}\log^{10393/900}T,
\end{equation}
where $B= 57.8875$ and $\mathcal{C}_1^{\prime}$ is the constant defined in Theorem \ref{theorem1}.
The inequality \eqref{mybetter} holds if and only if 
\[
T^{\frac{8}{3}(1-\sigma)}\ge \frac{\mathcal{C}_1^\prime}{C}T^{B(1-\sigma)^{3/2}}(\log T)^{7693/900}.
\]
Applying the logarithm to both sides the previous inequality becomes
\[
\frac{8}{3}(1-\sigma)\log T\ge \log{\frac{\mathcal{C}_1^\prime}{C}}+B(1-\sigma)^{3/2}\log T+\frac{7693}{900}\log\log T.
\]
Hence, for every fixed $T\ge 3$, our estimate in Theorem \ref{theorem1} is sharper when 
\[
\sigma\le 1- \frac{1}{\left(\frac{8}{3}-B(1-\sigma)^{1/2}\right)}\left(\frac{\log(\mathcal{C}_1^\prime/C)}{\log T}+\frac{7693\log\log T}{900\log T}\right),
\]
or, since the quantity $B(1-\sigma)^{1/2}$ is negligible when $\sigma$ is really close to $1$, for 
\begin{equation}\label{rangesigma}
  \sigma\le1-\frac{3}{8}\left(\frac{\log(\mathcal{C}_1^\prime/C)}{\log T}+\frac{7693\log\log T}{900\log T}\right).  
\end{equation}
Furthermore, we also want to see for which $T$ Theorem \ref{theorem1} is sharper than \eqref{kadiri} uniformly inside the range $\sigma\in[0.98,1)$. In order to that, we find the range of $T$ for which the current best-known zero-free regions for $\zeta$ have non-empty intersection with the region \eqref{rangesigma}. We start considering the best-known Korobov-Vinogradov zero-free region \eqref{kvzerofree} \cite{bellotti2023explicit}, which is the sharpest for $T\ge e^{481958}$. The intersection of the two regions is non-empty if
\[
 1-\frac{1}{53.989\log^{2/3}T(\log\log T)^{1/3}}\le 1-\frac{3}{8}\left(\frac{\log(\mathcal{C}_1^\prime/C)}{\log T}+\frac{7693\log\log T}{900\log T}\right)
\]
i.e.
\[
\frac{1}{53.989\log^{2/3}T(\log\log T)^{1/3}}\ge \frac{3}{8}\left(\frac{\log(\mathcal{C}_1^\prime/C)}{\log T}+\frac{7693\log\log T}{900\log T}\right).
\]
Using $\mathcal{C}_1^\prime=4.72\cdot 10^{20}$ and $C\ge 1$, the inequality is true for $T\ge \exp(6.7\cdot 10^{12})$. Since the range of values of $T$ for which Theorem \ref{theorem1} is sharper is included in the range of values for which the Korobov-Vinogradov zero-free region is the widest one, we do not need to check the intersection of the region \eqref{rangesigma} with the classical zero-free region and the Littlewood one.\\
Furthermore, following exactly the proof of Theorem \ref{theorem1}, the constant $\mathcal{C}_1^{\prime}$ can be further improved for $T\ge \exp(6.7\cdot 10^{12}) $.
\begin{theorem}\label{theorem2}
   For $T\ge \exp(6.7\cdot 10^{12})$ we have
   \[
   N(\sigma,T)\le 4.45\cdot 10^{12}\cdot T^{57.8875(1-\sigma)^{3/2}}(\log T)^{10393/900}.
   \]
\end{theorem}
We notice that our estimate beats Kadiri-Lumley-Ng's result for values of $T$ quite large. A major impediment is the power of the log-factor, i.e $10393/900\sim 11.547$ which is much larger than the log-power $\sim 3$ in \cite{KADIRI201822}. To get an explicit estimate of the form \eqref{ivicest} which is always sharper than \eqref{kadiri}, one should get in Theorem \ref{theorem1} a log-power less or at most equal to $3$, as in this case the term $T^{B(1-\sigma)^{3/2}}$ would be always dominant on $T^{8/3(1-\sigma)}$, being $(1-\sigma)<1$. However, as we already explained before, the current argument which uses Ivic's zero-detection method prevent us from getting $(\log T)^3$ in Theorem \ref{theorem1}, having a contribution of a log-factor at least $(\log T)^{9}$ based on the previous analysis. One way to overcome the logarithm problem is trying to find an explicit log-free zero density estimate (see \cite{montgomery_topics_1971}, $\S12$ for non-explicit results), which hopefully will improve \cite{KADIRI201822} when $T$ is relatively small. This might be also some interesting material for future work.\\
However, an improvement in Richert's bound and, as a consequence, an improved Korobov-Vinogradov zero-free region would lead anyway to an improved zero-density estimate of this type, getting a better constant $\mathcal{C}_1^{\prime}$, a better value of $B<57.8875$ in \eqref{ivicest} and a wider range of values of $T$ and $\sigma$ for which a zero-density estimate of the type \eqref{ivicest} is sharper than \eqref{kadiri}.
\section{Background}
We recall some results that will be useful for the proof of Theorem \ref{theorem1}.\\ First of all we give an overview of the currently best-known explicit zero-free regions for the Riemann zeta function. As we already mention, Platt and Trudgian \cite{platt_riemann_2021} verified that the Riemann Hypothesis is true up to height $3\cdot 10^{12}$. Hence, in the proof of Theorem \ref{theorem1} we will work with $T\ge 3\cdot 10^{12}$. Now, we list the best-known zero-free regions for $\zeta$. All of them hold for every $|T|\ge 3$, but we indicate the range of values for $|T|$ for which each of them is the widest one.
\begin{itemize}
    \item For $3\cdot 10^{12}< |T|\le e^{46.2}$ the largest zero-free region\cite{mossinghoff_explicit_2022} is the classical one:
    \begin{equation}\label{classicalzerofree}
        \sigma\ge 1-\frac{1}{5.558691 \log |T|}
    \end{equation}
    \item For $ e^{46.2}<|T|\le e^{170.2}$ the currently sharpest known zero-free region is \begin{equation}\label{zerofreeintermediate}
\sigma>1-\frac{0.04962-0.0196 /(J(|T|)+1.15)}{J(|T|)+0.685+0.155 \log \log |T|},
\end{equation}
where $J(T):=\frac{1}{6} \log T+\log \log T+\log 0.618$. As per \cite{yang2023explicit}, this result is obtained by substituting Theorem 1.1 of \cite{hiary2022improved} in Theorem 3 of \cite{ford_zero_2022} and observing that $J(T)<\frac{1}{4} \log T+1.8521$ for $T\geq 3$.
\item For $e^{170.2}< |T|\le e^{481958}$, Littlewood zero-free region \cite{yang2023explicit} becomes the largest one\footnote{The constant $21.432$ in \cite{yang2023explicit} has been updated to $21.233$ (communicated by the author, who has sent a proof of the revised Littlewood zero-free region).}:
\begin{equation}\label{littlewoodzerofree}
    \sigma\ge 1-\frac{\log\log |T|}{21.233\log |T|}.
\end{equation}
\item Finally, for $|T|> e^{481958}$, Korobov--Vinogradov zero-free region \cite{bellotti2023explicit} is the widest one\footnote{The value $54.004$ in \cite{bellotti2023explicit} can be improved to $53.989$ due to the improved estimate \eqref{littlewoodzerofree}.}: \begin{equation}\label{kvzerofree}
     \sigma\ge  1-\frac{1}{53.989\log^{2/3}|T|(\log\log |T|)^{1/3}}.
\end{equation}
\end{itemize}
The method used to detect the Korobov--Vinogradov zero-free region involves Richert's bound $|\zeta(\sigma+it)|\le A |t|^{B(1-\sigma)^{3/2}}(\log |t|)^{2/3}$, which will also be an essential tool in our proof of Theorem \ref{theorem1}, as we already mentioned. Below, we state the best-known estimate of this type (see \cite{ford_vinogradovs_2002} for previous results).
\begin{theorem}[\cite{bellotti2023explicit} Th.1.1]\label{bound}
    The following estimate holds for every $|t|\ge 3$ and $\frac{1}{2}\le \sigma\le 1$:
\begin{equation*}
    \begin{aligned}
    |\zeta(\sigma+it)|&\le A |t|^{B
    (1-\sigma)^{3/2}}\log^{2/3}|t|\\
        \left|\zeta(\sigma+i t, u)-u^{-s}\right| &\leq A |t|^{B(1-\sigma)^{3 / 2}} \log ^{2 / 3} |t|, \qquad 0<u \leq 1,
    \end{aligned}
\end{equation*}with $A=70.6995$ and $B=4.43795$.
\end{theorem}
Related to the zeros of $\zeta$, we state the best-known explicit estimate for the quantity $N(T)$, which is the number of zeros $\rho=\beta+i\gamma$ of $\zeta(s)$ with $0\le \gamma\le T$. 
\begin{theorem}[\cite{HASANALIZADE2022219} Corollary 1.2]\label{thnt}
  For any $T\ge e$ we have  \begin{equation*}
\left|N(T)-\frac{T}{2 \pi} \log \left(\frac{T}{2 \pi e}\right)\right| \leq 0.1038 \log T+0.2573 \log \log T+9.3675.
\end{equation*}
\end{theorem}
A powerful tool that is widely used in Ivić's zero-detection method is the Halász--Montgomery inequality we recall below.
\begin{theorem}[\cite{Ivic2003TheRZ} A.39, A.40]\label{montineq}
  Let $\xi,\varphi_1,\dots,\varphi_R$ be arbitrary vectors in an inner-product vector space over $\mathbb{C}$, where $(a,b)$ will be the notation for the inner product and $||a||^2=(a,a)$. Then
  \[
  \sum_{r\le R}|(\xi,\varphi_r)|\le ||\xi||\left(\sum_{r,s\le R}|(\varphi_r,\varphi_s)|\right)^{1/2}
  \]
  and
   \[
  \sum_{r\le R}|(\xi,\varphi_r)|\le ||\xi||^2\max_{r\le R}\sum_{s\le R}|(\varphi_r,\varphi_s)|.
  \]
\end{theorem}
Then, the following result on the divisor function will be useful.
\begin{theorem}[\cite{CullyHugill2019TwoED} Theorem 2]\label{michaela-tim}
For $x \geq 2$ we have
$$
\sum_{n \leq x} d(n)^2=D_1 x \log ^3 x+D_2 x \log ^2 x+D_3 x \log x+D_4 x+\vartheta\left(9.73 x^{\frac{3}{4}} \log x+0.73 x^{\frac{1}{2}}\right)
$$
where
$$
D_1=\frac{1}{\pi^2}, \quad D_2=0.745 \ldots, \quad D_3=0.824 \ldots, \quad D_4=0.461 \ldots
$$
are exact constants. Furthermore, for $x \geq x_j$ we have
$$
\sum_{n \leq x} d(n)^2 \leq K x \log ^3 x
$$
where one may take $\left\{K, x_j\right\}$ to be, among others, $\left\{\frac{1}{4}, 433\right\}$ or $\{1,7\}$.
\end{theorem}
In particular, for our purpose, when $x=10^{85}$ we have
\begin{equation*}
   \sum_{n \leq x} d(n)^2 \leq 0.106 x \log ^3 x. 
\end{equation*}
Finally, we recall an explicit estimate for the $\Gamma $ function (\cite{olver_asymptotics_1974}, p.294).
\begin{lemma}(Explicit Stirling formula)\label{stirling}
For $z=\sigma+it$, with $|\arg z| <\pi$ we have
\[
\left|\Gamma(z)\right|\le (2\pi)^{1/2}|t|^{\sigma-\frac{1}{2}}\exp\left(-\frac{\pi}{2}|t|+\frac{1}{6|z|}\right).
\]
\end{lemma}
\section{Proof of Theorem \ref{theorem1general}}
We will follow Ivić's zero-detection method (\cite{Ivic2003TheRZ}, $\S11$) using near-optimal choices. As in \cite{Ivic2003TheRZ}, we start considering the following relation derived by a well-known Mellin Transform:
\[
e^{-n/Y}=\frac{1}{2\pi i}\int_{2-i\infty}^{2+i\infty}\Gamma(w)Y^wn^{-w}dw.
\]
Then, we consider the function 
\[
M_{X}(s)=\sum_{n\le X}\frac{\mu(n)}{n^s},\qquad s=\sigma+it,\qquad \log T\le |t|\le T, \ 1\ll X\ll Y\ll T^{c},
\]
where $X=X(T)$ and $Y=Y(T)$ are parameters we will choose later. Furthermore, by our later choices, we will have
\begin{equation}\label{rangex}
    X\ge\left\{\begin{array}{lll}
     10^{85} &  \text{if }3\cdot10^{12}< T \le e^{46.2},\\ \\
        10^{89} &   \text{if }e^{46.2}< T \le e^{170.2}, \\ \\
        10^{100} &   \text{if }e^{170.2}< T \le e^{481958}, \\ \\
        10^{165} &   \text{if } T >  e^{481958}.
    \end{array}\right.
\end{equation}
By the elementary relation 
\[\sum_{d|n}\mu(d)=\left\{\begin{array}{cc}
    1 &   \text{if }n=1\\
     0 & \text{if }n>1,
\end{array}\right.\]
we see that each zero $\rho=\beta +i\gamma$ of $\zeta(s)$ counted by $N(\sigma,T)$ satisfies
\begin{equation}\label{firstrel}
    \begin{aligned}
        e^{-1/Y}+\sum_{n>X}a(n)n^{-\rho}e^{-n/Y}=\frac{1}{2\pi i}\int_{2-i\infty}^{2+i\infty}\zeta(\rho+w)M_{X}(\rho+w)Y^w\Gamma(w)dw
    \end{aligned}
\end{equation}
with\[
a(n)=\sum_{d|n,\ d\le X}\mu(d),\qquad |a(n)|\le d(n)<n^\varepsilon.
\]
As per \cite{Ivic2003TheRZ}, for a fixed zero $\rho=\beta+i\gamma$, we move the line of integration to $\operatorname{Re}w=\alpha-\beta<0$ for some suitable $\frac{1}{2}\le \alpha\le 1$ . As per Ivić \cite{Ivic2003TheRZ}, the optimal choice for $\alpha$ is $\alpha=5\sigma-4$. Hence, since in the hypothesis of Theorem 
\ref{theorem1} we assumed $\sigma\in[0.98,1)$, it follows that $\alpha\ge 0.9$. If $|\gamma|>2\log T$, the poles we find are at $w=0$ and $w=1-\rho$. Hence, using the residue theorem, the relation \eqref{firstrel} becomes
\begin{equation}\label{mainrel}
\begin{aligned}
    & e^{-1/Y}+\sum_{n>X}a(n)n^{-\rho}e^{-n/Y}=\frac{1}{2\pi i}\int_{2-i\infty}^{2+i\infty}\zeta(\rho+w)M_{X}(\rho+w)Y^w\Gamma(w)dw\\&=\zeta(\rho)M_X(\rho)+M_{X}(1)Y^{1-\rho}\Gamma(1-\rho)+\frac{1}{2\pi i}\int_{\alpha-\beta-i\infty}^{\alpha-\beta+i\infty}\zeta(\rho+w)M_{X}(\rho+w)Y^w\Gamma(w)dw\\&=M_{X}(1)Y^{1-\rho}\Gamma(1-\rho)+\frac{1}{2\pi i}\int_{-\infty}^{+\infty}\zeta(\alpha+i\gamma+iv)M_{X}(\alpha+i\gamma+iv)\Gamma(\alpha-\beta+iv)Y^{\alpha-\beta+iv}\text{d}v.
\end{aligned}\end{equation} Furthermore, we split both the integral and the sum in \eqref{mainrel} into the following terms:
\begin{equation}\begin{aligned}\label{spllitint}
    &\int_{-\infty}^{+\infty}\zeta(\alpha+i\gamma+iv)M_{X}(\alpha+i\gamma+iv)\Gamma(\alpha-\beta+iv)Y^{\alpha-\beta+iv}\text{d}v\\&=\int_{-\log T}^{\log T}\zeta(\alpha+i\gamma+iv)M_{X}(\alpha+i\gamma+iv)\Gamma(\alpha-\beta+iv)Y^{\alpha-\beta+iv}\text{d}v\\&+\int_{|v|\ge \log T}\zeta(\alpha+i\gamma+iv)M_{X}(\alpha+i\gamma+iv)\Gamma(\alpha-\beta+iv)Y^{\alpha-\beta+iv}\text{d}v
\end{aligned}\end{equation}
and 
\[
\sum_{n>X}a(n)n^{-\rho}e^{-n/Y}=\sum_{X<n\le Y\log Y}a(n)n^{-\rho}e^{-n/Y}+\sum_{n> Y\log Y}a(n)n^{-\rho}e^{-n/Y}.
\]
Finally, we define the following quantities: 

\begin{equation}\label{defterms}
    \begin{aligned}
    A&=\sum_{X<n\le Y\log Y}a(n)n^{-\rho}e^{-n/Y},\\
    B&=\frac{1}{2\pi i}\int_{-\log T}^{\log T}\zeta(\alpha+i\gamma+iv)M_{X}(\alpha+i\gamma+iv)\Gamma(\alpha-\beta+iv)Y^{\alpha-\beta+iv}\text{d}v,\\
     D&=\frac{1}{2\pi i}\int_{|v|\ge \log T}\zeta(\alpha+i\gamma+iv)M_{X}(\alpha+i\gamma+iv)\Gamma(\alpha-\beta+iv)Y^{\alpha-\beta+iv}\text{d}v\\&+M_{X}(1)Y^{1-\rho}\Gamma(1-\rho)-\sum_{n>Y\log Y}a(n)n^{-\rho}e^{-n/Y}.   
    \end{aligned}
\end{equation}
\begin{remark}
   The choice of splitting the integral \eqref{spllitint} in $|v|<\log T$ and $|v|\ge \log T$ is the optimal one. Indeed, in order to get a final estimate as sharp as possible, the power of $\log T$ in the estimate of the quantity $B$ should be the smallest possible. However, if one would take $(\log T)^{1-\epsilon}$, convergence problems arise in the estimate of the quantity $D$.
\end{remark}
Using the above notation, the relation \eqref{mainrel} can be rewritten as \[
e^{-1/Y}+A=B+D.
\]
Also,
\[
B-A=e^{-1/Y}-D\geq 1-\frac{1}{Y}+\frac{1}{2Y^2}> 0.
\]
As we will see in Lemma \ref{termd}, the quantity $D\rightarrow 0$ as $Y\rightarrow \infty$, hence for $Y\rightarrow +\infty$  one has $B-A\rightarrow 1$. It follows that at least one between the quantities $A$ and $B$ must be $\gg 1$. More precisely, given $0<c<1$, we have $|B|\ge c$ or, if $|B|<c$, then $|A|\ge 1-\epsilon-c$, where 
\[
\epsilon=\frac{1}{Y}-\frac{1}{2Y^2}+D.
\]
Putting $c=1-\epsilon-c$, the optimal bound for both $|A|$ and $|B|$ is
\[
0.49999\le c_0=\frac{1}{2}(1-\epsilon)=\frac{1}{2}\left(1-\frac{1}{Y}+\frac{1}{2Y^2}-D\right)\le \frac{1}{2},
\]where we recall that $Y\ge X\ge 10^{85}$.\\
It follows that each $\rho_r=\beta_r+i\gamma_r$, $\beta_r\ge \sigma$ counted by $N(\sigma,T)$ satisfies at least one of the following conditions:
\begin{equation}\label{cond1}
    \left|\sum_{X<n\le Y\log  Y}a(n)n^{-\sigma-i\gamma_r}\right|\ge c_0,
\end{equation}
\begin{equation}\label{cond2}
    \left|\int_{-\log T}^{\log T}\zeta(\alpha+i\gamma_r+iv)M_{X}(\alpha+i\gamma_r+iv)\Gamma(\alpha-\beta+iv)Y^{\alpha-\beta+iv}\text{d}v\right|\ge c_0
\end{equation}
or
\begin{equation}\label{cond3}
 |\gamma_r|\le 2\log T.   
\end{equation}
We recall that the coefficients $a(n)$ in \eqref{cond1} satisfy $|a(n)|\le |d(n)|$.\\
The number of zeros $\rho$ satisfying \eqref{cond3} is

\begin{equation*}
\begin{aligned}
    &2N(2\log T)\\&\le 2\left(\frac{2\log T}{2 \pi} \log \left(\frac{2\log T}{2 \pi e}\right)+  0.1038 \log (2\log T)+0.2573 \log \log (2\log T)+9.3675\right)\\&\le 0.45\log T\log\log T
\end{aligned}
\end{equation*}

where we used Theorem \ref{thnt} and the fact that $T>3\cdot10^{12}$.\\
Then, we denote with $R_1$ the number of zeros satisfying \eqref{cond1} and with $R_2$ the number of zeros satisfying \eqref{cond2} so that the imaginary parts of these zeros differ from each other by at least $2\log^{1.4} T$. We have
\begin{equation}\label{finalzerod}
    N(\sigma,T)\le \left(R_1+R_2+1\right)0.45\log ^{1.4}T\log\log T.
\end{equation}
\begin{remark}
    The constraint $|\gamma_r-\gamma_s|>2\log^{1.4}T$, where $\gamma_r,\gamma_s$ are distinct zeros satisfying \eqref{cond1} or \eqref{cond2}, is necessary for convergence issues we will find later in the estimate of both $R_1$ and $R_2$. A smaller exponent for the log-factor would cause convergence problems.
\end{remark}
Before proceeding with the estimate for $R_1$ and $R_2$, we prove that $D\rightarrow 0$ as $Y\rightarrow \infty$, as we mentioned before, where $D$ is defined in \eqref{defterms}.
\begin{lemma}\label{termd}
    Under the above assumptions, the relation $D\le 10^{-5}$ holds. Furthermore, $D\rightarrow 0$ as $Y\rightarrow \infty$.
\end{lemma}
\begin{proof}
    We will estimate each term of $D$ in \eqref{defterms} separately.\\Using Lemma \ref{stirling} with $z=\alpha-\beta+iv$, we have
\[
\left|\Gamma(\alpha-\beta+iv)\right|\le |v|^{\alpha-\beta-\frac{1}{2}}e^{-\pi |v|/2}(2\pi)^{1/2}e^{1/(6|v|)}=(2\pi)^{1/2}|v|^{\alpha-\beta-\frac{1}{2}}\exp\left(-\frac{\pi}{2}|v|+\frac{1}{6|v|}\right).
\]
This estimate, together with Theorem \ref{bound} and the relation $|\gamma+v|\le 2T\le 2e^{|v|}$, which holds for $|v|\ge \log T$,
gives the following estimate for the first term in $D$:
\begin{equation}\label{err2}
    \begin{aligned}
        &\left|\frac{1}{2\pi i}\int_{|v|\ge \log T}\zeta(\alpha+i\gamma+iv)M_{X}(\alpha+i\gamma+iv)\Gamma(\alpha-\beta+iv)Y^{\alpha-\beta+iv}\text{d}v\right|\\&\le 31.09\cdot \frac{1}{Y^{\beta-\alpha}}\int_{|v|\ge \log T}  e^{4.43795
    (1-\alpha)^{3/2}|v|}(\log(2e^{|v|}))^{2/3}|v|^{\alpha-\beta-\frac{1}{2}}\exp\left(-\frac{\pi}{2}|v|+\frac{1}{6|v|}\right)\text{d}v\\&\le  62.2\cdot \frac{e^{\frac{1}{6|\log T|}}}{Y^{\beta-\alpha}}\int_{|v|\ge \log T} e^{4.43795
    (1-\alpha)^{3/2}|v|-\frac{\pi|v|}{2}}  |v|^{\frac{2}{3}+\alpha-\beta-\frac{1}{2}}\text{d}v\\&\le  124.4\cdot 10^{-18}\cdot \frac{e^{\frac{1}{6|\log T|}}}{Y^{\beta-\alpha}}\\&\le 10^{-12}.
    \end{aligned}
\end{equation}
\begin{color}{cyan}
    
\end{color}Furthermore, for  $Y\rightarrow \infty$, and hence $T\rightarrow \infty$, the estimate found in \eqref{err2} goes to $0$.\\Now, we estimate the second term in $D$. Using again Lemma \ref{stirling} but with $z=1-\rho$ we have
\begin{equation}\label{err1}
\left|\Gamma(1-\beta-i\gamma)\right|\le |\gamma|^{1-\beta-\frac{1}{2}}e^{-\pi |\gamma|/2}(2\pi)^{1/2}e^{1/(6|\gamma|)}=(2\pi)^{1/2}|\gamma|^{1-\beta-\frac{1}{2}}\exp\left(-\frac{\pi}{2}|\gamma|+\frac{1}{6|\gamma|}\right).
\end{equation}
Since we are dealing with a zero $\rho$ such that $|\gamma|>2\log T$, if we use the estimate \eqref{caseclassic} for $Y$, it follows that
\begin{equation}\label{err1}
\left|M_{X}(1)Y^{1-\rho}\Gamma(1-\rho)\right|\le \frac{(2\pi)^{1/2}e^{\frac{1}{6|\gamma|}}Y^{1-\beta}}{e^{\frac{\pi |\gamma|}{2}}|\gamma|^{\beta-1+\frac{1}{2}}}\le  10^{-10}.
\end{equation}
As before, we notice that for $Y\rightarrow \infty$, and hence $T\rightarrow \infty$, the estimate \eqref{err1} goes to $0$. Finally,
\begin{equation}\label{err3}
    \begin{aligned}
        \left|\sum_{n>Y\log Y}a(n)n^{-\rho}e^{-n/Y}\right|&\le \sum_{n>Y\log Y}\left|d(n)n^{-\beta}e^{-n/Y}\right|\\&\le (Y\log Y)^{\epsilon-\beta}\left|\int_{Y\log Y}^{\infty}e^{-u/Y}du\right|\\&\le (Y\log Y)^{\epsilon-\beta}\frac{Y}{e^{\log Y}}\\&=(Y\log Y)^{\epsilon-\beta}
    \end{aligned}
\end{equation}
for any $\epsilon>0$ arbitrarily small. Taking $\epsilon $ so that $\epsilon-\beta<0$ ($\epsilon=10^{-N}$ with $N$ large, for example), it follows that \eqref{err3} is less than $10^{-10}$ and it goes to $0$ as $Y\rightarrow \infty$.\\This concludes the proof.
\end{proof}
\begin{remark}
    The choice $|\gamma|>2\log T$ is near-optimal. Indeed, in order to get the best possible final estimate, we need the smallest admissible power of $\log T$, but if we choose $|\gamma|>\log^{1-\epsilon} T$, with $\epsilon >0$ arbitrary small, or $|\gamma|>c\log T$, with $c<2$, the estimate in \eqref{err1} is not $o(1)$ anymore.
\end{remark}
We now proceed with the estimate for both $R_1$ and $R_2$. 
\subsection{Estimate for $R_1$}
First of all, by dyadic division there exists a number $M$ such that $X\le M\le Y\log Y$ and
\begin{equation}\label{ineqr}
    \begin{aligned}
        \left|\sum_{M<n\le 2M}a(n)n^{-\sigma-i\gamma_r}\right|&\ge \frac{c_0}{\frac{\log(Y)+\log(\log Y)-\log X}{\log 2}-1}\\&=\frac{1}{\log Y}\frac{c_0}{\frac{1+(\log(\log Y)/\log Y)-(\log X/\log Y)}{\log 2}-\frac{1}{\log Y}}\\&=\frac{1}{\log Y}C_1,
    \end{aligned}
\end{equation}
where
\begin{equation}\label{valueC1}
  C_1=\frac{c_0}{\frac{1+(\log(\log Y)/\log Y)-(\log X/\log Y)}{\log 2}-\frac{1}{\log Y}}\ge  \left\{\begin{array}{lll}
      0.3386 &  \text{if }3\cdot10^{12}< T \le e^{46.2},\\ \\
      0.3389   &   \text{if }e^{46.2}< T \le e^{170.2}, \\ \\
           0.3395      &   \text{if }e^{170.2}< T\le e^{481958}, \\ \\
       0.3418 &   \text{if }T >  e^{481958}. 
         \end{array}\right.
\end{equation}

The number of zeros $R$ satisfying the inequality \eqref{ineqr} will be $R\ge R_1C_1/(c_0\log Y)$.\\
Now, we apply the Halász--Montgomery inequality in Theorem \ref{montineq} with $\xi=\{\xi_n\}_{n=1}^\infty$ such that
\[\left\{\begin{array}{ll}
  \xi_n=a(n)(e^{-n/2M}-e^{-n/M})^{-1/2}n^{-\sigma}   & \text{if }M<n\le 2M \\ \\
   \xi_n=0  & \text{otherwise},
\end{array}\right.
\]
and
\[
\varphi_{r}=\{\varphi_{r,n}\}_{n=1}^{\infty},\quad \varphi_{r,n}=(e^{-n/2M}-e^{-n/M})^{1/2}n^{-it_r}\qquad \forall n=1,2,\dots.
\]
We have\begin{equation}\label{estforR}
    \begin{aligned}
        R^2&\le \frac{1}{C^2_1} \log^2Y\left(\sum_{M<n\le 2M}a(n)^2e^{-2n/Y}n^{-2\sigma}\right)\left(RM+\sum_{r\neq s\le R}|H(it_r-it_s)|\right),
    \end{aligned}
\end{equation}
where
\[
H(it)=\sum_{n=1}^\infty(e^{-n/2M}-e^{-n/M})n^{-it}=\frac{1}{2\pi i}\int_{2-i\infty}^{2+i\infty}\zeta(w+it)((2M)^w-M^w)\Gamma(w)dw.
\]
Moving the line of integration in the expression for $H(it)$ to $\operatorname{Re}w=\alpha$, we encounter a pole in $w=1-it$ and, by the residue theorem we get
\begin{equation}\label{esth}
    \begin{aligned}
        H(it)&=((2M)^{1-it}-M^{1-it})\Gamma(1-it)+\frac{1}{2\pi i}\int_{\alpha-i\infty}^{\alpha+i\infty}\zeta(w+it)((2M)^w-M^w)\Gamma(w)dw.
    \end{aligned}
\end{equation}
Combining \eqref{esth} with \eqref{estforR} one has
\begin{align*}
        &R^2\le \frac{1}{C^2_1} \log^2Y\left(\sum_{M<n\le 2M}d(n)^2e^{-2n/Y}n^{-2\sigma}\right)\left(RM+\sum_{r\neq s\le R}|H(it_r-it_s)|\right)\\&\le\frac{1}{C^2_1} \log^2Y\left(\sum_{M<n\le 2M}d(n)^2e^{-2n/Y}n^{-2\sigma}\right)\times\\&\left(RM+\sum_{r\neq s\le R}\left|((2M)^{1-i(t_r-t_s)}-M^{1-i(t_r-t_s)})\Gamma(1-i(t_r-t_s))\right|\right.\\&\left.+\sum_{r\neq s\le R}\left|\frac{1}{2\pi i}\int_{\alpha-i\infty}^{\alpha+i\infty}\zeta(w+i(t_r-t_s))((2M)^w-M^w)\Gamma(w)dw\right|\right),
    \end{align*}and, splitting the integral in two parts, the above inequality becomes
    \begin{equation}\label{R1}
        \begin{aligned}
            &\le \frac{1}{C^2_1} \log^2Y\left(\sum_{M<n\le 2M}d(n)^2e^{-2n/Y}n^{-2\sigma}\right)\times\\&\left(RM+M\sum_{r\neq s\le R}\left|M^{-i(t_r-t_s)}(2^{1-i(t_r-t_s)}-1)\Gamma(1-i(t_r-t_s))\right|\right.\\&+\frac{M^\alpha}{2\pi }\sum_{r\neq s\le R}\left|\int_{-\log^{1.5} T}^{\log^{1.5} T}\zeta(\alpha+it_r-it_s+iv)M^{iv}(2^{iv}-1)\Gamma(\alpha+iv)\text{d}v\right|\\&\left.+\frac{M^\alpha}{2\pi }\sum_{r\neq s\le R}\left|\int_{|v|\ge \log^{1.5}T}\zeta(\alpha+it_r-it_s+iv)M^{iv}(2^{iv}-1)\Gamma(\alpha+iv)\text{d}v\right|\right).
        \end{aligned}
    \end{equation}
    \begin{remark}
   The choice of splitting the integral in $|v|<\log^{1.5} T$ and $|v|\ge \log^{1.5} T$ is near-optimal. Indeed, lower powers of $\log T$ give better estimates, but if one would decrease the log-power only to $\log^{1.4} T$, the term \eqref{thirdterm}is not $o(1) $ sufficiently small anymore.
\end{remark}
Now, we start estimating the quantity \[
\left(\sum_{M<n\le 2M}d(n)^2e^{-2n/Y}n^{-2\sigma}\right).
\]
Using Theorem \ref{michaela-tim}, we get
\begin{equation}
    \begin{aligned}
        &\left(\sum_{M<n\le 2M}d(n)^2e^{-2n/Y}n^{-2\sigma}\right)\le 0.106M^{-2\sigma}e^{-2M/Y}(2M\log^3(2M)-M\log^3M)\\&\le 0.106 M^{1-2\sigma}e^{-2M/Y} 1.021\log^3M\le 0.109M^{1-2\sigma}e^{-2M/Y}\log^3M,
    \end{aligned}
\end{equation}

since  $M\ge X\ge 10^{85}$ by assumption.\\Furthermore, using Lemma \ref{stirling}, it follows that 
\[
\left|\Gamma(\alpha+iv)\right|\le |v|^{\alpha-\frac{1}{2}}e^{-\pi |v|/2}(2\pi)^{1/2}e^{1/(6|v|)}=(2\pi)^{1/2}|v|^{\alpha-\frac{1}{2}}\exp\left(-\frac{\pi}{2}|v|+\frac{1}{6\alpha}\right)
\]
and, similarly,
\[
\left|\Gamma(1-i(t_r-t_s))\right|\le (2\pi)^{1/2}|t_r-t_s|^{1-\frac{1}{2}}\exp\left(-\frac{\pi}{2}|t_r-t_s|+\frac{1}{6}\right).
\]
Then, \eqref{R1} becomes
\begin{equation}\label{R11}
    \begin{aligned}
        &R^2\le \frac{0.109}{C_1^2}\cdot  M^{1-2\sigma}e^{-2M/Y}\log^5M \times\\&\left(RM+2M(2\pi)^{1/2}\sum_{r\neq s\le R}|t_r-t_s|^{\frac{1}{2}}\exp\left(-\frac{\pi}{2}|t_r-t_s|+\frac{1}{6}\right)\right.\\&+\frac{M^\alpha}{\pi }\sum_{r\neq s\le R}\int_{-\log^{1.5} T}^{\log^{1.5} T}\left|\zeta(\alpha+it_r-it_s+iv)\Gamma(\alpha+iv)\right|\text{d}v\\&\left.+\frac{M^\alpha}{\pi }\sum_{r\neq s\le R}\left|\int_{|v|\ge \log^{1.5} T}\zeta(\alpha+it_r-it_s+iv)\Gamma(\alpha+iv)\text{d}v\right|\right).
    \end{aligned}
\end{equation}
From now on, we will denote with $C_3$ the following quantity:
\begin{equation*}
    C_3=\frac{0.109}{C^2_1}\le\left\{\begin{array}{lll}
    0.9503 &  \text{if }3\cdot10^{12}< T \le e^{46.2},\\ \\
      0.9488 &   \text{if }e^{46.2}< T \le e^{170.2}, \\ \\
        0.9453 &   \text{if }e^{170.2}< T\le e^{481958}, \\ \\
        0.9327 &   \text{if }T>  e^{481958},
    \end{array}\right.
\end{equation*}
where we used the bounds found in \eqref{valueC1}.\\
Now, we denote 
\begin{equation}
    \mathcal{M}(\alpha,T):=70.6995T^{4.43795(1-\alpha)^{3/2}}(\log T)^{2/3}.
\end{equation}
By Theorem \ref{bound}, we know that
\begin{equation}\label{boundzeta}
    \max_{|t|\le T}|\zeta(\alpha+it)|\le \mathcal{M}(\alpha,T).
\end{equation}

Furthermore, the function 
\[
\sqrt{y}\exp\left(-\frac{\pi}{2}y\right)
\] is decreasing in $y$, hence, since $|t_r-t_s|\ge \log^{1.4} T$, the estimate \eqref{R11} becomes
\begin{equation}\label{genineq}
    \begin{aligned}
        &\le\frac{C_3}{e^{2M/Y}}  M^{1-2\sigma}\log^5M\left(RM+2e^{\frac{1}{6}}R^2M(2\pi)^{1/2}\log^{1.4/2} T\exp\left(-\frac{\pi}{2}\log^{1.4} T\right)\right.\\&+\frac{\sqrt{2}M^\alpha}{\sqrt{\pi} }\sum_{r\neq s\le R}\int_{-\log^{1.5} T}^{\log^{1.5} T}\left|\zeta(\alpha+it_r-it_s+iv)\right||v|^{\alpha-\frac{1}{2}}\exp\left(-\frac{\pi}{2}|v|+\frac{1}{6\alpha}\right)\text{d}v\\&\left.+\frac{M^\alpha}{\pi }\sum_{r\neq s\le R}\left|\int_{|v|\ge \log^{1.5} T}\zeta(\alpha+it_r-it_s+iv)\Gamma(\alpha+iv)\text{d}v\right|\right).
    \end{aligned}
\end{equation}
Before estimating all the terms on the right side of the inequality \eqref{genineq}, we define the following quantity:
\begin{equation}\label{dfx}
    X=\left(D_1\mathcal{M}(\alpha,3T)\log^5T\right)^{1/(2\sigma-1-\alpha)},
\end{equation}
where $D_1=1.01\cdot 10^{12}$. Since $X\le M$ by assumption, we have \[
\mathcal{M}(\alpha,3T)\le \frac{M^{2\sigma-1-\alpha}}{D_1\log^5T}.
\]
It follows that
\begin{equation}\label{maintermwithr}
    \begin{aligned}
        &\frac{\sqrt{2}C_3M^{1-2\sigma+\alpha}\log^5M}{e^{2M/Y}\sqrt{\pi} } \sum_{r\neq s\le R}\int_{-\log^{1.5} T}^{\log^{1.5} T}\left|\zeta(\alpha+it_r-it_s+iv)\right||v|^{\alpha-\frac{1}{2}}\exp\left(-\frac{\pi}{2}|v|+\frac{1}{6\alpha}\right)\text{d}v\\&\le \frac{\sqrt{2}C_3M^{1-2\sigma+\alpha}\log^5Me^{\frac{1}{6\alpha}}R^2}{e^{2M/Y}\sqrt{\pi}}\mathcal{M}(\alpha,3T)\int_{-\log^{1.5} T}^{\log^{1.5} T}|v|^{\alpha-\frac{1}{2}}\exp\left(-\frac{\pi}{2}|v|\right)\text{d}v\\&\le \frac{\sqrt{2}C_3e^{\frac{1}{6\alpha}}R^2M^{1+\alpha-2\sigma}}{\sqrt{\pi}e^{2M/Y} }\mathcal{M}(\alpha,3T)\log^5M\\&\le \frac{\sqrt{2}C_3}{e^{2M/Y}}\frac{e^{\frac{1}{6\alpha}}R^2}{\sqrt{\pi}D_1}\frac{\log^5M}{\log^5T}.
    \end{aligned}
\end{equation}
Furthermore, by Theorem \ref{bound}, the relation $|t_r-t_s+v|\le 3T\le 3e^{|v|^{1/1.5}}$ which holds for $|v|\ge (\log T)^{1.5}$ and the fact that $\alpha\ge 0.9$ as we explained before, we have
\begin{equation}\label{thirdterm}
\begin{aligned}
    &\frac{C_3}{e^{2M/Y}}\frac{M^{\alpha+1-2\sigma}\log^5M}{\pi }\sum_{r\neq s\le R}\left|\int_{|v|\ge \log^{1.5} T}\zeta(\alpha+it_r-it_s+iv)\Gamma(\alpha+iv)\text{d}v\right|\\&\le 26.26\cdot \frac{C_3}{e^{2M/Y}}\cdot M^{\alpha+1-2\sigma}\log^5M\\&\ \ \times \sum_{r\neq s\le R}\int_{|v|\ge \log^{1.5} T}  e^{4.43795
    (1-\alpha)^{3/2}|v|^{1/1.5}}\log(3e^{|v|^{1/1.5}})^{2/3}|v|^{\alpha-\frac{1}{2}}\exp\left(-\frac{\pi}{2}|v|+\frac{1}{6|v|}\right)\text{d}v\\&\le 157.8\cdot \frac{C_3}{e^{2M/Y}} \frac{e^{\frac{1}{6\log^{1.5} T}}\log^5M}{e^{0.01 \pi\log^{1.5} T}} \sum_{r\neq s\le R}\int_{v\ge \log^{1.5} T}  e^{4.43795
    (1-\alpha)^{3/2}v^{1/1.5}-0.49\pi v}v^{\frac{2}{ 4.5}+\alpha-\frac{1}{2}}\text{d}v\\&\le 157.8\cdot 10^{-99}\cdot \frac{C_3}{e^{2M/Y}} R^2\frac{e^{\frac{1}{6\log^{1.5} T}}\log^5M}{e^{0.01 \pi\log^{1.5} T}}.
\end{aligned}
\end{equation}
Now, we define the following quantity $Y$:
\begin{equation}\label{dfy}
Y=\left\{D_2 \mathcal{M}(\alpha, 3 T)\right\}^{(3 \sigma-2 \alpha-1) /(\sigma-\alpha)(2 \sigma-1-\alpha)}(\log T)^{(-\frac{1661}{300} \sigma-\frac{1661}{300}\alpha+\frac{1661}{150}) / 2(\sigma-\alpha)(2 \sigma-1-\alpha)},
\end{equation}
where $D_2=7.26\cdot 10^6$.
\begin{remark}
    The exponent of the log-factor in the definition of $Y$ is the near-optimal one, and allows us to get a final lower exponent for $\log T$, compared to Ivić's result in \cite{Ivic2003TheRZ}. Furthermore, we already optimize the constants $D_1$ and $D_2$ in the definition of $X$ and $Y$ respectively.
\end{remark}
Using \eqref{dfy} and $\alpha=5\sigma-4$, we have, by Theorem \ref{bound},
\[
M\le Y\log Y,\qquad Y\le D_2^{\frac{7}{12(1-\sigma)}}70.6995^{\frac{7}{12(1-\sigma)}}(3T)^{28.9437\sqrt{1-\sigma}}(\log 3T)^{\frac{1661}{1200(1-\sigma)}}.
\]Now, we want to estimate $\log Y$ in the four different ranges of values for $T$.
\begin{itemize}
    \item \textbf{Case $3\cdot 10^{12}< T\le e^{46.2}$}. By \eqref{classicalzerofree}, we can assume $$
\sigma< 1-\frac{1}{5.558691 \log T}.
$$
We have
\begin{equation}\label{caseclassic}
\begin{aligned}
   & \log Y\le \log\left(D_2^{\frac{7}{12(1-\sigma)}}70.6995^{\frac{7}{12(1-\sigma)}}(3T)^{28.9437\sqrt{1-\sigma}}(\log 3T)^{\frac{1661}{1200(1-\sigma)}}\right)\\&\le \frac{7}{12(1-\sigma)}20.057+4.094\log(3T)+\frac{1661}{1200(1-\sigma)}\log\log(3T)\\&\le 65.066\log T+4.251\log T+29.673\log T\\&\le 98.99  \log T.
\end{aligned}
\end{equation}
\item \textbf{Case $e^{46.2}<T\le e^{170.2}$} By \eqref{zerofreeintermediate}, in this range we will work with \begin{equation}\label{zerofreeintermediate}
\sigma\le 1-\frac{0.04962-0.0196 /(J(T)+1.15)}{J(T)+0.685+0.155 \log \log T},
\end{equation}
where $J(T):=\frac{1}{6} \log T+\log \log T+\log 0.618$.\\
We get
\begin{equation}
\begin{aligned}
   & \log Y\le \log\left(D_2^{\frac{7}{12(1-\sigma)}}70.6995^{\frac{7}{12(1-\sigma)}}(3T)^{28.9437\sqrt{1-\sigma}}(\log 3T)^{\frac{1661}{1200(1-\sigma)}}\right)\\&\le \frac{7}{12(1-\sigma)}20.057+4.094\log(3T)+\frac{1661}{1200(1-\sigma)}\log\log(3T)\\&\le 65.069\log T+4.121\log T+29.674\log T\\&\le 98.864\log T.
\end{aligned}
\end{equation}      
\item \textbf{Case $e^{170.2}< T\le e^{481958}$}. By \eqref{littlewoodzerofree}, in this range we will assume \[
\sigma< 1-\frac{\log\log T}{21.233\log T}.
\]
In this case we have
\begin{equation}
\begin{aligned}
   & \log Y\le \log\left(D_2^{\frac{7}{12(1-\sigma)}}70.6995^{\frac{7}{12(1-\sigma)}}(3T)^{28.9437\sqrt{1-\sigma}}(\log 3T)^{\frac{1661}{1200(1-\sigma)}}\right)\\&\le \frac{7}{12(1-\sigma)}20.057+4.094\log(3T)+\frac{1661}{1200(1-\sigma)}\log\log(3T)\\&\le 48.382\log T+4.121\log T+29.427\log T\\&\le 81.93\log T.
\end{aligned}
\end{equation}
\item \textbf{Case $T> e^{481958}$}. By \eqref{kvzerofree}, we assume
    \[
    \sigma< 1-\frac{1}{53.989\log^{2/3}T(\log\log T)^{1/3}}.
    \]
    We have
    \begin{equation}
\begin{aligned}
   & \log Y\le \log\left(D_2^{\frac{7}{12(1-\sigma)}}70.6995^{\frac{7}{12(1-\sigma)}}(3T)^{28.9437\sqrt{1-\sigma}}(\log 3T)^{\frac{1661}{1200(1-\sigma)}}\right)\\&\le \frac{7}{12(1-\sigma)}20.057+4.094\log(3T)+\frac{1661}{1200(1-\sigma)}\log\log(3T)\\&\le 19.023\log T+4.095\log T+29.392\log T\\&\le 52.51\log T.
\end{aligned}
\end{equation}
\end{itemize}

Since $M\le Y\log Y\le Y^2$, \eqref{maintermwithr} becomes
\begin{equation}\label{firstapp}
    \le 2.427\cdot 10^{11}\frac{C_3}{e^{2M/Y}}\frac{e^{\frac{1}{6\alpha}}R^2}{D_1}.
\end{equation}
Also,
\begin{equation}\label{secondapp}
   \frac{\log^5M}{e^{0.01\pi\log^{1.5}T}}\le \frac{\log^5(Y\log Y)}{e^{0.01\pi\log^{1.5}T}}\le \frac{(2\log Y)^5}{e^{0.01\pi\log^{1.5}T}}\le 10^{18},
\end{equation}
and hence \eqref{thirdterm} becomes
\[
\le R^2 \cdot  10^{-70}.
\]
\begin{remark}
    In order to estimate both \eqref{firstapp} and \eqref{secondapp}, we used the worst estimate for $\log Y$ among the four found above, since the contribution is negligible.
\end{remark}
Furthermore
\begin{equation}
\begin{aligned}
    &\frac{2C_3}{e^{2M/Y}}(2\pi)^{1/2}e^{1/6}M^{2-2\sigma}\log^5M\log^{1.4/2}T\exp\left(-\frac{\pi}{2}\log^{1.4}T\right)\\&\le 2(2\pi)^{1/2}e^{1/6}C_3 M^{0.04}(2\cdot 98.99\log T)^5\log^{1.4/2}T\exp\left(-\frac{\pi}{2}\log^{1.4}T\right)\\&\le  C_3 2.18\cdot 10^{12}\cdot T^{3.96}(\log T)^{5.74}\exp\left(-\frac{\pi}{2}\log^{1.4}T\right)\le 10^{-4} .
\end{aligned}
\end{equation}
\begin{remark}
    The choice $|t_r-t_s|\ge \log^{1.4} T$ is almost optimal, as otherwise even with $1.3$ the above quantity is not small enough.
\end{remark}
Hence, dividing by $R$ we get
\begin{equation}
    \begin{aligned}
        &R \le \frac{C_3}{e^{2M/Y}}M^{2-2\sigma}\log^5M+R\cdot 10^{-4}+2.427\cdot 10^{11}\cdot\frac{C_3e^{\frac{1}{6\alpha}}R}{D_1}+R\cdot 10^{-70},
    \end{aligned}
\end{equation}

or equivalently
\begin{equation}
    \begin{aligned}
    RC_2\le  \frac{C_3}{e^{2M/Y}}M^{2-2\sigma}\log^5M,
    \end{aligned}
\end{equation}
where\begin{align*}
    C_2&=\left(1- 10^{-4}-2.427\cdot 10^{11}\cdot\frac{C_3e^{\frac{1}{6\alpha}}}{D_1}- 10^{-70}\right)\\&\ge \left(1- 10^{-4}-2.427\cdot 10^{11}\cdot\frac{C_3e^{\frac{1}{6}}}{D_1}- 10^{-70}\right)\\&\ge\left\{\begin{array}{lll}
    0.7301 &  \text{if }3\cdot10^{12}< T\le e^{46.2},\\ \\
       0.7305 &   \text{if }e^{46.2}< T \le e^{170.2}, \\ \\
       0.7315 &   \text{if }e^{170.2}< T\le e^{481958}, \\ \\
       0.7351 &   \text{if }T> e^{481958}.
    \end{array}\right.
\end{align*}
Now, for $\sigma\ge (\alpha+1)/2$,
we have 
\begin{equation}
    \begin{aligned}
        &R_1\le\max_{X\le M=2^k\le Y\log Y}  C_4 \frac{M^{2-2\sigma}}{e^{2M/Y}}\log^5M\log Y\\&\le   C_4 Y^{2-2\sigma}(2\log Y)^5\log Y\frac{(1-\sigma)^{2-2\sigma}}{e^{2-2\sigma}}\\&\le 2^5\cdot C_4\cdot  Y^{2-2\sigma}\log^6Y,
    \end{aligned}
\end{equation}
where 
\[
C_4\le \frac{C_3c_0}{C_2\cdot C_1}\le \left\{\begin{array}{lll}
    1.9213 &  \text{if }3\cdot10^{12}< T \le e^{46.2},\\ \\
     1.9157   &   \text{if }e^{46.2}< T \le e^{170.2}, \\ \\
     1.9024 &   \text{if }e^{170.2}< T \le  e^{481958}, \\ \\
      1.8557 &   \text{if }T>  e^{481958}.
    \end{array}\right.,
\]
and we used the fact that the maximum of the function for the variable $M$
\[
\frac{M^{2-2\sigma}}{e^{2M/Y}}
\]is reached in $M=Y(1-\sigma)$. 
It follows that
\begin{equation}
    R_1\le C_5\cdot Y^{2-2\sigma}\log^6T,
\end{equation}
where, using the upper bounds found for $\log Y$ in the three different ranges, 
\[
C_5\le \left\{\begin{array}{lll}
    5.785\cdot 10^{13} &  \text{if }3\cdot10^{12}< T\le e^{46.2},\\ \\
    5.724\cdot 10^{13}    &   \text{if }e^{46.2}< T \le e^{170.2}, \\ \\
        1.842\cdot 10^{13} &   \text{if }e^{170.2}< T\le e^{481958}, \\ \\
       1.245\cdot 10^{12} &   \text{if }T>  e^{481958}.
    \end{array}\right.
\]
Finally,
\begin{equation}
\begin{aligned}
    &Y^{2-2\sigma}\le D_2^{7/6}70.6995^{7/6}(3T)^{57.8875(1-\sigma)^{3/2}}(\log(3T))^{7/9}(\log T)^{1661/600}.
    \end{aligned}
\end{equation}
It follows that
\begin{equation}\label{finalestimateforr1}
   R_1\le \mathcal{C}\cdot T^{57.8875(1-\sigma)^{3/2}}(\log T)^{17183/1800}, 
\end{equation}
where
\[
\mathcal{C}\le \left\{\begin{array}{lll}
     1.04\cdot 10^{24} &  \text{if }3\cdot10^{12}< T\le e^{46.2},\\ \\
      1.02\cdot 10^{24}   &   \text{if }e^{46.2}< T \le e^{170.2}, \\ \\
        3.22\cdot 10^{23} &   \text{if }e^{170.2}< T\le  e^{481958}, \\ \\
       2.17\cdot 10^{22} &   \text{if }T>  e^{481958}.
    \end{array}\right.
\]
\subsection{Estimate for $R_2$} Contrary to what we did for estimating $R_1$, for $R_2$ we will just work with the Korobov--Vinogradov zero-free region, which is asymptotically the widest one. Indeed, the contribution given by the quantity $R_2$ relies primarily on the power of the log-factor in the final result, which is not affected by the choice of the zero-free region we are working with. Indeed, the difference between the slightly sharper contribution from the optimal zero-free region for each $T$ and that one obtained with the Korobov--Vinogradov zero-free region for all $T$ is negligible. Hence, from now on we will just work with the zero-free region \eqref{kvzerofree}, which, as we already mentioned before, holds for every $|T|\ge 3$.\\
Given \begin{equation}
    \int_{-\log T}^{\log T}\zeta(\alpha+i\gamma_r+iv)M_{X}(\alpha+i\gamma_r+iv)\Gamma(\alpha-\beta+iv)Y^{\alpha-\beta+iv}\text{d}v\ge c_0,
\end{equation}
we denote with $t_r$ a real number such that $|t_r|\le 2T$, $|t_r-t_s|>\log^{1.4} T$ for $r\neq s$, and such that the quantity
\[
\left|\zeta(\alpha+it_r)M_{X}(\alpha+it_r)\right|
\]
is maximum. Furthermore, since $\beta\ge\sigma\ge \frac{1+\alpha}{2}$ and $\alpha=5\sigma-4$, we have, for every $T> 3\cdot 10^{12}$,
\begin{align*}
  \beta-\alpha\ge \sigma-5\sigma+4=4-4\sigma\ge
       \frac{4}{(53.989)(\log T)^{2 / 3}(\log \log T)^{1 / 3}}.
\end{align*}
In order to estimate the quantity
\begin{equation}\label{intgamma}
     \int_{-\log T}^{\log T}|\Gamma(\alpha-\beta+iv)|\text{d}v,
\end{equation}
we observe that the functional equation for the Gamma function $\Gamma(z+1)=z\Gamma(z)$ implies
\[
|\Gamma(\alpha-\beta+iv)|=\frac{1}{|\alpha-\beta+iv|}|\Gamma(\alpha-\beta+1+iv)| \le \frac{1}{|\alpha-\beta+iv|}|\Gamma(\alpha-\beta+1)|.
\]
Since by our hypotheses on $\alpha$ and $\beta$ we have $0<\alpha-\beta+1<1$, we can apply to $\Gamma(\alpha-\beta+1)$ the following relation for $\Gamma$ which holds for every real $0<z<1$, i.e.,
\begin{equation}\label{uppergamma}
\Gamma(z) = \frac{1}{z}\int_0^{\infty}x^ze^{-x}\text{d}x < \frac{1}{z}\int_0^{\infty}xe^{-x}\text{d}x = \frac{1}{z},  
\end{equation}
obtaining 
\begin{equation}\label{ineqgamma}
  |\Gamma(\alpha-\beta+iv)|\le  \frac{1}{|\alpha-\beta+iv||\alpha-\beta+1|}.  
\end{equation}
Now, we define the following quantity:
\[
\mathscr{A}:=\frac{4}{53.989(\log T)^{2/3}(\log(\log T))^{1/3}}.
\] 
The integral \eqref{intgamma} can be rewritten as
\begin{equation}
   \int_{-\log T}^{\log T}|\Gamma(\alpha-\beta+iv)|\text{d}v=2 \int_{0}^{\mathcal{A}}|\Gamma(\alpha-\beta+iv)|\text{d}v+2 \int_{\mathcal{A}}^{\log T}|\Gamma(\alpha-\beta+iv)|\text{d}v,
\end{equation}
since $\mathcal{A}\le \log T$ for every $T>3\cdot 10^{12}$. We estimate the two terms separately.\\
Since $|\alpha - \beta + iv| \ge \mathcal{A}$, it follows that
\begin{equation}\label{part1}
    2 \int_{0}^{\mathcal{A}}|\Gamma(\alpha-\beta+iv)|\text{d}v \le 2 \int_{0}^{\mathcal{A}}\frac{\text{d}v}{|\alpha-\beta+iv||\alpha-\beta+1|}\le \frac{2}{\mathcal{A}|\alpha-\beta+1|}\cdot \mathcal{A}\le 2.06,
\end{equation}
where we used the inequality \eqref{ineqgamma}.\\ However, when $|v| > \mathcal{A}$, we also have $|\alpha - \beta + iv| \ge |v|$. Hence, using \eqref{ineqgamma}, one gets
\begin{equation}\label{part2}
\begin{aligned}
     &2 \int_{0}^{\mathcal{A}}|\Gamma(\alpha-\beta+iv)|\text{d}v \le 2 \int_{0}^{\mathcal{A}}\frac{\text{d}v}{|\alpha-\beta+iv||\alpha-\beta+1|}\\&\le  \frac{1}{|\alpha - \beta + 1|}\int_{\mathcal{A}}^{\log T}\frac{\text{d}v}{v} \le 2.06\log(\log T).
\end{aligned}   
\end{equation}
By \eqref{part1} and \eqref{part2} it follows that 
\begin{align*}
    &\int_{-\log T}^{\log T}\left|\Gamma(\alpha-\beta+iv)\right|\text{d}v\le 2.7\log(\log T).
\end{align*}
Now,
\begin{align*}
&\int_{-\log T}^{\log T}\zeta(\alpha+i\gamma_r+iv)M_{X}(\alpha+i\gamma_r+iv)\Gamma(\alpha-\beta+iv)Y^{\alpha-\beta+iv}\text{d}v\\&\le 2.7\log(\log T) Y^{\alpha-\sigma}\mathcal{M}(\alpha,2T)M_X(\alpha+it_r).
\end{align*}
 Hence, one has
\[
M_{X}(\alpha+it_r)\ge \frac{Y^{\sigma-\alpha}c_0}{2.7\log(\log T)\mathcal{M}(\alpha,2T)}
\]
for $R_2$ points $t_r$ such that $|t_r|\le 2T$, $|t_r-t_s|>\log^{1.4} T$ for $r\neq s\le R_2$.\\
Then, there exists a number $N\in[1, X]$, such that
\begin{align*}
    &\sum_{N<n\le 2N}\mu(n)n^{-\alpha-it_r}\\&\ge \frac{c_0 Y^{\sigma-\alpha}}{2.7\log(\log T)\mathcal{M}(\alpha,2T)\log T}\frac{1}{\frac{1}{\log 2}-\frac{1}{\log X}}\\&\ge\frac{Y^{\sigma-\alpha}}{\mathcal{M}(\alpha,2T)\log(\log T)\log T}C_6,
\end{align*}
with \[
C_6=\frac{c_0}{2.7}\frac{1}{\frac{1}{\log 2}-\frac{1}{\log X}}\ge \frac{0.49999\log 2}{2.7}\ge 0.128,
\]
for $R_0\ge R_2/(d\log T)$ with $d=\frac{1}{\log 2}-\frac{1}{\log X}\le 1.45$ numbers $t_r$.\\As in the case for $R_1$, we apply Halász--Montgomery inequality in Theorem \ref{montineq} with
\[
\xi=\{\xi_n\}_{n=1}^\infty,\quad \xi_n=\mu(n)(e^{-n/2N}-e^{-n/N})^{-1/2}n^{-\alpha},\quad N<n\le 2N
\]
and $0$ otherwise. Also,
\[
\varphi_{r}=\{\varphi_{r,n}\}_{n=1}^{\infty},\quad \varphi_{r,n}=(e^{-n/2M}-e^{-n/M})^{1/2}n^{-it_r}.
\]We get
\begin{align*}
        &R_0\le \mathcal{M}^2(\alpha,2T)(\log\log T)^2(\log T)^2 Y^{2\alpha-2\sigma}\frac{1}{C^2_6}\left(\sum_{N<n\le 2N}\mu(n)^2e^{-2n/X}n^{-2\alpha}\right)\\&\ \times \left(R_0N+\sum_{r\neq s\le R_0}|H(it_r-it_s)|\right)\\&\le  \mathcal{M}^2(\alpha,2T)(\log\log T)^2(\log T)^2 Y^{2\alpha-2\sigma}\frac{1}{C^2_6}\left(\sum_{N<n\le 2N}\mu(n)^2e^{-2n/X}n^{-2\alpha}\right)\\&\left(R_0N+\sum_{r\neq s\le R_0}\left|((2N)^{1-i(t_r-t_s)}-N^{1-i(t_r-t_s)})\Gamma(1-i(t_r-t_s))\right|\right.\\&\left.+\sum_{r\neq s\le R_0}\left|\frac{1}{2\pi i}\int_{\alpha-i\infty}^{\alpha+i\infty}\zeta(w+i(t_r-t_s))((2N)^w-N^w)\Gamma(w)dw\right|\right),
        \end{align*}
        and, splitting the integral in two terms, the previous relation becomes
\begin{equation}\label{eqr2}
            \begin{aligned}
      &\le \mathcal{M}^2(\alpha,2T)(\log\log T)^2(\log T)^2 Y^{2\alpha-2\sigma}\frac{1}{C^2_6}\left(\sum_{N<n\le 2N}\mu(n)^2e^{-2n/X}n^{-2\alpha}\right)\\&\left(R_0N+N\sum_{r\neq s\le R_0}\left|N^{-i(t_r-t_s)}(2^{1-i(t_r-t_s)}-1)\Gamma(1-i(t_r-t_s))\right|\right.\\&+\frac{N^\alpha}{2\pi }\sum_{r\neq s\le R_0}\left|\int_{-\log^{2} T}^{\log^{2} T}\zeta(\alpha+it_r-it_s+iv)N^{iv}(2^{iv}-1)\Gamma(\alpha+iv)\text{d}v\right|\\&\left.+\frac{N^\alpha}{2\pi }\sum_{r\neq s\le R_0}\left|\int_{|v|\ge \log^{2} T}\zeta(\alpha+it_r-it_s+iv)N^{iv}(2^{iv}-1)\Gamma(\alpha+iv)\text{d}v\right|\right).
    \end{aligned}
\end{equation}
We want to estimate all the terms in the above inequality.\\ First of all, a trivial estimate gives
\[
\left(\sum_{N<n\le 2N}\mu(n)^2e^{-2n/X}n^{-2\alpha}\right)\le N^{1-2\alpha}e^{-2N/X}.
\]
Then, following exactly the same argument as for $R_1$, together with the upper bound for the Gamma function in Lemma \ref{stirling}, the previous inequality \eqref{eqr2} becomes
    \begin{align*}
        &\le \mathcal{M}^2(\alpha,2T)(\log\log T)^2(\log T)^2 Y^{2\alpha-2\sigma}\frac{1}{C^2_6} N^{1-2\alpha}e^{-2N/X}\\&\ \  \times \left(R_0N+2e^{\frac{1}{6}}NR_0^2 (2\pi)^{1/2}\log^{1.4/2}T\exp\left(-\frac{\pi}{2}\log^{1.4}T\right)\right.\\&\quad +\ \left.\frac{\sqrt{2}e^{\frac{1}{6\alpha}}R_0^2N^\alpha}{\sqrt{\pi} }\mathcal{M}(\alpha,3T)\right.\\&\left.\quad+\frac{N^{\alpha}}{\pi }\sum_{r\neq s\le R_0}\int_{|v|\ge \log^{2} T}|\zeta(\alpha+it_r-it_s+iv)||\Gamma(\alpha+iv)|\text{d}v\right)\\&\le \mathcal{M}^2(\alpha,2T)(\log\log T)^2(\log T)^2 Y^{2\alpha-2\sigma}\frac{1}{C^2_6} e^{-2N/X}\\&\ \ \times\left(R_0N^{2-2\alpha}+2R_0^2N^{2-2\alpha}e^{\frac{1}{6}}(2\pi)^{1/2}\log^{1.4/2}T\exp\left(-\frac{\pi}{2}\log^{1.4}T\right)\right.\\& \left.\quad +\frac{\sqrt{2}e^{\frac{1}{6\alpha}}R_0^2N^{1-\alpha}}{\sqrt{\pi} }\mathcal{M}(\alpha,3T)\right.\\&\left.\quad+\frac{N^{1-\alpha}}{\pi }\sum_{r\neq s\le R_0}\int_{|v|\ge \log^{2} T}|\zeta(\alpha+it_r-it_s+iv)||\Gamma(\alpha+iv)|\text{d}v\right).
    \end{align*}
In order to estimate the last term in the previous inequality, since the relation $|t_r-t_s+v|\le 5T\le 5e^{|v|^{1/2}}$ holds for $|v|\ge \log^2T$, we observe that
\begin{equation}\label{caser2small}
\begin{aligned}
    & \mathcal{M}^2(\alpha,2T)(\log\log T)^2(\log T)^2 \frac{Y^{2\alpha-2\sigma}N^{1-\alpha}}{C^2_6\pi e^{2N/X} }\\&\ \ \times \sum_{r\neq s\le R_0}\int_{|v|\ge \log^{2} T}|\zeta(\alpha+it_r-it_s+iv)||\Gamma(\alpha+iv)|\text{d}v\\&\le \frac{82.48e^{\frac{1}{6\log^{2} T}}}{C_6^2}\mathcal{M}^2(\alpha,2T)N^{1-\alpha} Y^{2\alpha-2\sigma}(\log\log T)^2(\log T)^2\\&\ \ \times\sum_{r\neq s\le R_0}\int_{|v|\ge \log^{2} T}  e^{4.43795
    (1-\alpha)^{3/2}|v|^{1/2}-\frac{\pi}{2}|v|}|v|^{\frac{1}{3}+\alpha-\frac{1}{2}}\text{d}v\\&\le \frac{164.96}{C_6^2}\frac{e^{\frac{1}{6\log^{2} T}}\mathcal{M}^2(\alpha,2T)N^{1-\alpha} Y^{2\alpha-2\sigma}(\log\log T)^2(\log T)^2}{e^{0.01 \pi\log^{2} T}}\\&\ \ \times\sum_{r\neq s\le R_0}\int_{v\ge \log^{2} T}  e^{4.43795
    (1-\alpha)^{3/2}v^{1/2}-0.49\pi v}v^{\alpha-\frac{1}{6}}\text{d}v\\&\le \frac{164.96}{C^2_6} \cdot 10^{-92}R_0^2\frac{e^{\frac{1}{6\log^{2} T}}\mathcal{M}^2(\alpha,2T)N^{1-\alpha} Y^{2\alpha-2\sigma}(\log\log T)^2(\log T)^2}{e^{0.01 \pi\log^{2} T}}.
\end{aligned}
\end{equation}
Using the definition of $X$ in \eqref{dfx}, since $X\le Y$, $N\le X$ and $\alpha-\sigma<0$, we have
\begin{align*}
    &Y^{2\alpha-2\sigma}N^{1-\alpha}\mathcal{M}^2(\alpha,2T)\log^2T\le X^{\alpha+1-2\sigma}\mathcal{M}^2(\alpha,3T)\log^2T\le \frac{\mathcal{M}(\alpha,3T)}{D_1(\log T)^3}.
\end{align*}
Hence, \eqref{caser2small} becomes
\begin{equation}\label{p1}
\begin{aligned}
    &\le \frac{164.96}{D_1C^2_6} \cdot 10^{-92}R_0^2\frac{e^{\frac{1}{6\log^{2} T}}\mathcal{M}(\alpha,3T)(\log\log T)^2}{e^{0.01 \pi\log^{2} T}\log^3 T}\\&\le  \frac{11663}{C_6^2D_1}\cdot 10^{-92}R_0^2\frac{e^{\frac{1}{6\log^{2} T}}(3T)^{4.43795(1-\alpha)^{3/2}}\log^{2/3}(3T)(\log\log T)^2}{e^{0.01 \pi\log^{2} T}\log^3 T}\\&\le R_0^2 10^{-80}.
\end{aligned}
\end{equation}
Furthermore, we observe that
\begin{equation}\label{estr2fin}
\begin{aligned}
    &\mathcal{M}^2(\alpha,2T)Y^{2\alpha-2\sigma} X^{2-2\alpha}\log^3T\le \mathcal{M}^2(\alpha,2T)(\mathcal{M}(\alpha,3T))^{-4/3}D_2^{-14/3}D_1^{10/3}(\log T)^{1289/150}\\&\le (\mathcal{M}(\alpha,3T))^{2/3}D_2^{-14/3}D_1^{10/3}(\log T)^{1289/150}\\&\le C_9T^{33.08(1-\sigma)^{3/2}}\log^{4067/450}T,
\end{aligned}\end{equation}
with $$C_9=1.017\cdot (70.6995)^{2/3}\cdot 3^{\frac{2}{3}\cdot 4.43795(1-\alpha)^{3/2}}\cdot D_2^{-14/3}D_1^{10/3}\le 1.92\cdot 10^9.$$
Hence, 
\begin{equation}\label{p2}
\begin{aligned}
    &\frac{2}{C^2_6} e^{-2N/X}\mathcal{M}^2(\alpha,2T)Y^{2\alpha-2\sigma}R_0^2N^{2-2\alpha}e^{\frac{1}{6}}(2\pi)^{1/2}(\log\log T)^2\log^{2+1.4/2}T\exp\left(-\frac{\pi}{2}\log^{1.4}T\right)\\&\le \frac{2}{C^2_6} \mathcal{M}^2(\alpha,2T)Y^{2\alpha-2\sigma}X^{2-2\alpha}R_0^2e^{\frac{1}{6}}(2\pi)^{1/2}(\log\log T)^2\log^{2.7}T\exp\left(-\frac{\pi}{2}\log^{1.4}T\right)\\&\le R_0^210^{-30}.
\end{aligned}
\end{equation}
Finally, by definition of $X,Y$ \eqref{dfx}, \eqref{dfy}, we have
\begin{equation}
    \begin{aligned}
       & Y^{2\alpha-2\sigma}X^{1-\alpha}\mathcal{M}^3(\alpha,3T)\log^2T(\log\log T)^2\le Y^{2\alpha-2\sigma}X^{1-\alpha}\mathcal{M}^3(\alpha,3T)(\log T)^{2.74}\\& =D_1^{(1-\alpha)/(2\sigma-1-\alpha)}D_2^{-2(3\sigma-2\alpha-1)/(2\sigma-1-\alpha)}=D_1^{5/3}D_2^{-14/3},
    \end{aligned}
\end{equation}
where we used the inequality $\log\log T\le (\log T)^{0.37}$, which holds for $T>3\cdot 10^{12}$.
It follows that, being $N\le X$,
\begin{equation}\label{p3}
    \begin{aligned}
        &\frac{\sqrt{2}}{\sqrt{\pi}C^2_6} e^{\frac{1}{6\alpha}}R^2_0e^{-2N/X}Y^{2\alpha-2\sigma} N^{1-\alpha}\mathcal{M}^3(\alpha,3T)\log^2T(\log\log T)^2\\&\le \frac{\sqrt{2}}{\sqrt{\pi}C^2_6} e^{\frac{1}{6\alpha}}R^2_0 Y^{2\alpha-2\sigma} X^{1-\alpha}\mathcal{M}^3(\alpha,3T)(\log T)^{2.74}\\&\le \frac{\sqrt{2}e^{\frac{1}{6\alpha}}R^2_0}{\sqrt{\pi}C^2_6} D_1^{5/3}D_2^{-14/3}\\&\le R_0^2\cdot 10^{-10}.
    \end{aligned}
\end{equation}
Using the estimates \eqref{p1},\eqref{p2} and \eqref{p3} and dividing by $R_0$ we get
\begin{equation}\label{finalr0}
    \begin{aligned}
        &R_0\le \frac{1}{C^2_6} e^{-2N/X}\mathcal{M}^2(\alpha,2T)Y^{2\alpha-2\sigma} X^{2-2\alpha}(\log T)^2(\log\log T)^2+R_0( 10^{-30}+  10^{-10}+ 10^{-80}),
    \end{aligned}
\end{equation}
or equivalently,
\[
C_7R_0\le \frac{1}{C_6^2}e^{-2N/X}\mathcal{M}^2(\alpha,2T)Y^{2\alpha-2\sigma} X^{2-2\alpha}(\log T)^2(\log\log T)^2,
\]
where
\[
C_7=1-10^{-30}-  10^{-10}- 10^{-80}\ge 0.9999.
\]
It follows that
\[
R_0\le C_8 e^{-2N/X}\mathcal{M}^2(\alpha,2T)Y^{2\alpha-2\sigma} X^{2-2\alpha}(\log T)^2(\log\log T)^2,
\]
with 
\[
C_8=\frac{1}{C_7C_6^2}\le 61.05.
\]
Hence, for $\sigma\ge (\alpha+1)/2$, using \eqref{estr2fin} we have 
\begin{equation}
    \begin{aligned}
        R_2&\le\max_{1\le N=2^k\le X} d\cdot  C_8 e^{-2N/X}\mathcal{M}^2(\alpha,2T)Y^{2\alpha-2\sigma} X^{2-2\alpha}(\log T)^3(\log\log T)^2\\&\le d\cdot  C_8 \mathcal{M}^2(\alpha,2T)Y^{2\alpha-2\sigma} X^{2-2\alpha}(\log T)^3(\log\log T)^2\\&\le d\cdot  C_8 C_9T^{33.08(1-\sigma)^{3/2}}(\log T)^{4067/450}(\log T)^{0.74}\\&\le C_{10}T^{33.08(1-\sigma)^{3/2}}\log^{88/9}T,
    \end{aligned}
\end{equation}
where 
\[
C_{10}=dC_8C_9\le 1.7\cdot 10^{11}.
\]
We can conclude that
\begin{equation}\label{finalestimateforr2}
    R_2\le C_{10}T^{33.08(1-\sigma)^{3/2}}\log^{88/9}T.
\end{equation}
\subsection{Conclusion}
From \eqref{finalzerod}, we recall that
\begin{equation}
    N(\sigma,T)\le \left(R_1+R_2+1\right)0.45\log ^{1.4}T\log\log T.
\end{equation}
Inserting the estimates found for both $R_1$ and $R_2$, i.e. \eqref{finalestimateforr1} and \eqref{finalestimateforr2} respectively, we finally get
\begin{equation}\label{finalineq}
    \begin{aligned}
         &N(\sigma,T)\le \left(R_1+R_2+1\right)0.45\log ^{1.4}T\log\log T\\&\le \left(\mathcal{C} T^{57.8875(1-\sigma)^{3/2}}(\log T)^{17183/1800}+C_{10}T^{33.08(1-\sigma)^{3/2}}(\log T)^{88/9}+1\right)0.45(\log T)^{1.4}\log\log T\\&\le \mathcal{C}_1T^{57.8875(1-\sigma)^{3/2}}(\log T)^{19703/1800}\log\log T+\mathcal{C}_2T^{33.08(1-\sigma)^{3/2}}(\log T)^{503/45}\log\log T\\&\ \ +0.27(\log T)^{14/10}\log\log T,
    \end{aligned}
\end{equation}
where  $\mathcal{C}_2=7.65\cdot 10^{10}$ and
\begin{equation}
    \mathcal{C}_1=\left\{\begin{array}{lll}
    4.68\cdot 10^{23} &  \text{if }3\cdot10^{12}\le T\le e^{46.2},\\ \\
     4.59\cdot 10^{23}   &   \text{if }e^{46.2}< T \le e^{170.2}, \\ \\
        1.45\cdot 10^{23} &   \text{if }e^{170.2}< T\le e^{481958}, \\ \\
        9.77\cdot 10^{21} &   \text{if }T>  e^{481958}.
    \end{array}\right.
\end{equation}
This concludes the proof of Theorem \ref{theorem1general}.
\section{Proof of Theorem \ref{theorem1}}
Using the estimate in Theorem \ref{theorem1general} and the maximum between the exponents $19703/1800$ and $503/45$ of the log-factors, we have
\begin{equation}\label{simple}
    \begin{aligned}
          &N(\sigma,T)\le  \mathcal{C}^{\prime}_1 T^{57.8875(1-\sigma)^{3/2}}(\log T)^{\frac{503}{45}}\log\log T,
    \end{aligned}
\end{equation}
where
\begin{equation}\label{c1}
    \mathcal{C}^{\prime}_1=\left\{\begin{array}{lll}
   2.15\cdot 10^{23} &  \text{if }3\cdot10^{12}< T\le e^{46.2},\\ \\
    1.89\cdot 10^{23}    &   \text{if }e^{46.2}< T \le e^{170.2}, \\ \\
        4.42\cdot 10^{22} &   \text{if }e^{170.2}< T\le e^{481958}, \\ \\
        4.72\cdot 10^{20} &   \text{if } T>  e^{481958}.
    \end{array}\right.
\end{equation}
Finally, for $T\ge 3\cdot 10^{12}$, we recall that $\log\log T\le (\log T)^{0.37}$. Hence, \eqref{simple} becomes

\begin{equation}
    N(\sigma,T)\le \mathcal{C}^{\prime}_1T^{57.8875(1-\sigma)^{3/2}}(\log T)^{10393/900},
\end{equation}
where $\mathcal{C}_1^\prime$ is defined in \eqref{c1}.\\This concludes the proof of Theorem \ref{theorem1}.
\section*{Acknowledgements}
I would like to thank my supervisor Timothy S. Trudgian for his support and helpful suggestions throughout the writing of this article.
\clearpage
\printbibliography
\end{document}